\documentclass[12pt,reqno]{amsart}
\pdfoutput=1

\usepackage{latexsym}
\usepackage[centertags]{amsmath}
\usepackage{amsfonts}
\usepackage{amssymb}
\usepackage{amsthm}
\usepackage{newlfont}
\usepackage{graphics}
\usepackage{color}

\usepackage[usenames,dvipsnames]{xcolor}
\usepackage{graphicx}
\def\dessin#1#2{\includegraphics[#1]{#2}}
\usepackage{wrapfig}

\newtheorem{theo}{Theorem}
\newtheorem{defn}[theo]{Definition}

\newtheorem{lem} [theo]{Lemma}
\newtheorem{cor}[theo]{Corollary}
\newtheorem{prop}[theo]{Proposition}
\newtheorem{rem}[theo]{Remark}

\newtheorem{conj}[theo]{Conjecture}

\newcommand\qbinom[2]{\left[#1 \atop #2  \right]_q}
\def\n{\mathfrak{n}}
\def\N{\mathbb{N}}
\def\Z{\mathbb{Z}}

\newcommand{\area}{\operatorname{area}}
\newcommand{\coarea}{\operatorname{coarea}}
\newcommand{\dinv}{\operatorname{dinv}}
\newcommand{\codinv}{\operatorname{codinv}}
\newcommand{\arm}{\operatorname{arm}}
\newcommand{\leg}{\operatorname{leg}}

\newcommand{\Cat}{\texttt{Cat}}

\title[Rational Dyck paths and $(m,n)$-cores]{Rank complement of rational Dyck paths and conjugation of $(m,n)$-core partitions}

\author{Guoce Xin}
\address{School of Mathematical Sciences, Capital Normal University, Beijing 100048, PR China}
\email{guoce.xin@gmail.com}

\date{April 8, 2015} 

\begin{document}

\begin{abstract}
Given a coprime pair $(m,n)$ of positive integers, rational Catalan numbers $\frac{1}{m+n} \binom{m+n}{m,n}$ counts two combinatorial objects:
rational $(m,n)$-Dyck paths are lattice paths in the $m\times n$ rectangle that never go below the diagonal;
$(m,n)$-cores are partitions with no hook length equal to $m$ or $n$.
Anderson established a bijection between $(m,n)$-Dyck paths and $(m,n)$-cores. We define a new transformation, called rank complement, on rational Dyck paths. We show that rank complement corresponds to conjugation of $(m,n)$-cores under Anderson's bijection. This leads to: i) a new approach to characterizing $n$-cores; ii)
 a simple approach for counting the number of self-conjugate $(m,n)$-cores; iii) a proof of the equivalence of two conjectured combinatorial sum formulas, one over rational $(m,n)$-Dyck paths and the other over $(m,n)$-cores, for rational Catalan polynomials.
\end{abstract}

\maketitle

{\small \textbf{Keywords}: rational Dyck paths; core partitions; rational Catalan polynomials. 

\section{Introduction}
Rational Catalan numbers are defined for a coprime pair $(m,n)$ of positive integers by
$$\Cat_{m,n}= \frac{1}{m+n} \binom{m+n}{m,n}=\frac{1}{m}\binom{m+n-1}{m-1,n} =\frac{(m+n-1)!}{m!n!}.$$
It was known  to count the set $\cal D_{m,n}$ of rational $(m,n)$-Dyck paths
in the $m\times n$ rectangle (see, e.g., \cite{Bizley}). To be precise, $(m,n)$-Dyck paths (for general $(m,n)$)  are lattice paths from $(0,0)$ to $(m,n)$ that use unit steps $(1,0)$ or $(0,1)$ and never go below the main diagonal line $y=n x/m$.
The ordinary $(n,n)$-Dyck path is counted by the classical Catalan number $\Cat_{n+1,n}$, which is known \cite{Stanley-Catalan} to count more than 200 distinct families of combinatorial objects.
Since
$(n+1,n)$-Dyck paths are easily seen to be in bijection with $(n,n)$-Dyck paths, this case is always referred to as the classical case.

Rational Catalan numbers also appears in the context of partitions. Anderson established a bijection $\alpha$ from the set $\cal D_{m,n}$ of rational Dyck paths to the set $\cal P_{m,n}$ of $(m,n)$-cores, which are partitions having no hook lengths equal to $m$ or $n$. Recently, Armstrong et. al.
\cite{Armstrong-mn-core} established some interesting results and conjectures on $(m,n)$-cores. Some of the conjectures has been (partly) solved (See, e.g., \cite{chen-self-conjugate-st-cores}, \cite{Stanley-Zanello}, \cite{Johnson-average-mn-core}), some of them are still open.

Our first contribution is to define the rank complement transformation on $(m,n)$-Dyck paths in Section \ref{sec:Rank Complement}. Then we claim in Theorem \ref{t-mncore-conjugate} that under Ansderson's map, rank complement of $(m,n)$-Dyck paths corresponds to conjugation of $(m,n)$-cores. Conjugation of $(m,n)$-cores was also described using semimodules in \cite[Lemma 2.21]{Gorsky-Mazin2}. This correspondence was not known before. See Armstrong's talk \cite{Armstrong-slides} in 2012.
Using this correspondence, we give a simple proof of the following result of \cite{number of self-conjugate mn core partitions}.
\begin{theo}\label{c-num-self-conjugate-mncore}
  Let $(m,n)$ be a coprime pair. Then the number of self conjugate $(m,n)$-cores is given by
\begin{align}
\binom{\lfloor m/2\rfloor +\lfloor n/2\rfloor}{ \lfloor m/2\rfloor ,\lfloor n/2\rfloor}.
\end{align}
\end{theo}
The original proof of Theory \ref{c-num-self-conjugate-mncore} gives a direct bijection from self-conjugate $(m,n)$-cores to lattice paths from $(0,0)$ to $(\lfloor m/2 \rfloor, \lfloor n/2 \rfloor)$. It
is based on the fact that self conjugate partitions are uniquely determined by their hook lengths of the main diagonal. These main diagonal hook lengths are characterized by certain conditions for self conjugate $n$-cores (hence $(m,n)$-cores).

Rational Catalan polynomials are $q$-analogous of $\Cat_{m,n}$ defined by
$$\Cat_{m,n}(q)= \frac{1}{[m+n]_q} \qbinom{m+n}{m,n}=\frac{[m+n-1]_q!}{[m]_q![n]_q!}, $$
where $[a]_q=1+q+\cdots +q^{a-1}$ and $[a]_q! = [a]_q\cdots [1]_q$.
It is a nontrivial fact that $\Cat_{m,n}(q)$ is a polynomial in $q$ with $\N$ coefficients when $m$ and $n$ are coprime.
For this fact, Haiman gives an algebraic proof
 in \cite[Prop. 2.5.2]{Haiman} and also gives an algebraic interpretation for these polynomials
as the Hilbert series of a suitable quotient ring of a polynomial ring \cite[Prop. 2.5.3 and 2.5.4]{Haiman}. 
This polynomial also appears to be connected to certain modules arising in the theory of rational
Cherednik algebras. Recently, it is shown in \cite{Garsia-Leven-Wallach-Xin} that the modified version
$[\gcd(m,n)]_q \Cat_{m,n}(q)$ for
arbitrary $m$ and $n$ is also a polynomial in $q$ with $\N$ coefficients.

An interesting problem is to explain the fact by finding a combinatorial interpretation of the rational Catalan polynomial $\Cat_{m,n}(q)$, especially of the form
$$ \Cat_{m,n}(q)= \sum_{D\in  X_{m,n}}  q^{\mbox{stat}(D)},$$
where stat is certain statistic on $D$, and $X_{m,n}$ is a set counted by $\Cat_{m,n}$. Such a solution was only known for the classical case by MacMahon \cite{MacMahon} using the major index, and for the $m=kn+1$ case by \cite{Loehr-higher-qtCatalan} using the area and bounce statistic. In the general case, there are two conjectured formulas. One came up as a consequence of a more general rational Shuffle conjecture from algebraic combinatorics.
\begin{conj}
\label{conj-dinv-catalan}
 Let $(m,n)$ be a coprime pair of positive integers. Then we have
\begin{align}
\sum_{D \in \cal D_{m,n}} q^{\area(D)+\codinv(D)} =\Cat_{m,n}(q)= \frac{1}{[m+n]_q} \qbinom{m+n}{m,n}
\end{align}
where the sum ranges over all $(m,n)$-Dyck paths $D$, $\area(D)$ is the number of lattice squares between $D$ and the main diagonal, and $\codinv(D)$ is a Dyck path statistic that can be given a simple geometric construction.
\end{conj}
The conjecture was also made in \cite[Conjecture 6]{Amstrong-Catalan}, where dinv is denoted $h^+_{m,n}$.

The other conjecture was made in the context of partitions by  Armstrong et al. \cite{Armstrong-mn-core}, where they
introduced the statistics length $\ell(\lambda)$ and skew-length $s\ell(\lambda)$ of an $(m,n)$-core $\lambda$.
\begin{conj}\cite{Armstrong-mn-core}
\label{conj-mncore}
 Let $m$ and $n$ be coprime positive integers. Then we have
\begin{align}
\sum_{\lambda \in \cal P_{m,n}} q^{\ell(\lambda)+s\ell(\lambda)} =\Cat_{m,n}(q)= \frac{1}{[m+n]_q} \qbinom{m+n}{m,n},
\end{align}
where the sum ranges over all $(m,n)$-cores $\lambda$.
\end{conj}
Both conjectured formulas are specialization of rational $q,t$-Catalan numbers at $t=1/q$. The similarity of the two conjectures suggests that they are equivalent.
Indeed,
our major contribution is Theorem \ref{t-skewlength-codinv}, which shows that under Anderson's $\alpha$ map, $\codinv(D)$ of a Dyck paht $D$ is taken to $s\ell (\lambda)$
of the $(m,n)$-core $\lambda=\alpha(D)$. Since it is clear that $\area(D)=\ell(\lambda)$, we prove the equivalence of Conjectures \ref{conj-dinv-catalan} and \ref{conj-mncore}.

It is worth mentioning the well-known
Shuffle conjecture, which represents the Hilbert series of the diagonal Harmonics as a combinatorial sum over certain labeled Dyck paths, called parking functions. Shuffle conjecture is currently one of the major open problems in algebraic combinatorics. Since it was formulated in 2003 and published in \cite{Haglund-Shuffle}, little progress was made until a refinement of the conjecture was found 5 years later and published in \cite{classComp}. Three special cases of the refined shuffle conjecture have been proved in \cite{shuffle-he}, \cite{shuffle-hh}, \cite{shuffle-hhe}, but the method becomes more and more complicated for the general situation.

While the classical shuffle conjecture is too hard to attack, its natural extension, the rational shuffle conjecture, has been discovered recently. The conjecture was found to be connected with many other area of mathematics, such as: the Elliptic Hall Algebra of Burban-Shiffmann-Vasserot, the Algebraic Geometry of Springer Fibers of Hikita,
the Double Affine Hecke Algebras of Cherednik, the HOMFLY polynomials, and the truly fascinating Shuffle
Algebra of symmetric functions. Its specializations at $t = q^{-1}$ (including the classical cases) are still
open to this date. What makes this specialization particularly fascinating is that both sides of the stated
identities have combinatorial interpretations. As a consequence, Conjecture \ref{conj-dinv-catalan} is already quite challenging.

The paper is organized as follows. Section 1 is this introduction. Section \ref{sec:Rank Complement} introduces notations on rational Dyck paths and the rank complement transformation. In trying to understand its relation with conjugation of $(m,n)$-cores, we find
Theorem \ref{t-SH-complement} in Section \ref{sec:n-core}, which gives a relation between $n$-cores $\lambda$ and it conjugate $\lambda^T$.
This leads to a new approach to characterizing $n$-core partitions. In Section \ref{sec:conjugation-rank-complement}, we introduce Anderson's bijection between $(m,n)$-Dyck paths and $(m,n)$-cores and give a new proof of Theorem \ref{c-num-self-conjugate-mncore}. In Section \ref{sec:dinv-skew-length}, we introduce the dinv statistic and establish the equivalence of Conjectures \ref{conj-dinv-catalan} and \ref{conj-mncore}.

\section{Rational Dyck paths and the rank complement transformation \label{sec:Rank Complement}}
\subsection{Notations}
In this note, a path $P=p_1\cdots p_{n}$ is an $(N,E)$-sequence depicted in the plane as a sequence of lattice points $(P_0,\dots, P_{n})$ such that
$P_0=(0,0)$ (if not specified) and $P_i-P_{i-1}$ equals $(0,1)$ if $p_i=N$ and $(1,0)$ if $p_i=E$. Thus $N$ simply means going north and $E$ means going east. The product of two paths $P$ and $Q$
are the juxtaposition $PQ$ of the NE sequences of $P$ and $Q$. Or equivalently, $PQ$ is the path obtained by putting the starting point of $Q$ at the end point of $P$.

Let $m$ and $n$ be positive integers. We denote by $\cal F_{m,n}$ the set of (free) paths from $(0,0)$ to $(m,n)$.
An $(m,n)$-Dyck path $D$ is a path in $\cal F_{m,n}$ that never goes below the diagonal line $y=\frac{n}{m} x$. We denote by $\cal D_{m,n}$ the set of all such rational (slope) Dyck paths.
Define the rank of a lattice point $(a,b)$ by $r(a,b)=mb-na$. When $(m,n)$ is a coprime pair, all the lattice points $\{(a,b): 0\le a \le m, \ 0\le b \le n\}$ in the $m$ by $n$ rectangle are in bijection with their ranks except for  $r(0,0)=r(m,n)=0$. Thus we can encode a path $P\in \cal F_{m,n}$ by its rank set $r(P)=\{ r(P_0),\dots, r(P_{m+n-1}) \}= \{r(P_1),\dots, r(P_{m+n})\} $ of $m+n$ distinct ranks. The ranks can also be defined recursively by
$r(P_0)=0$ and for $i=1,\dots, m+n$, $r(P_i)=r(P_{i-1}) +m$ if $p_i=N$ and $r(P_i)=r(P_{i-1}) -n$ if $p_i=E$. Then $P\in \cal F_{m,n}$
is a Dyck path if and only if all its ranks are nonnegative.

Denote by $S(D)$ the set of south ends of $D$, i.e, starting points of an N step of $D$. By abuse of notation, it is convenient to denote also by $S(D)=\{ s_0,s_1,\dots, s_{n-1} \}$ the set of the ranks of the south ends of $D$. We claim that ${ S(D)} \bmod{n} =\{0,1,\dots, n-1\}$. This is because
two lattice points with the same $y$-coordinate must have the same rank when modulo $n$, and the $y$-coordinates of the $s_i$ are all different.
We can similarly define $E(D),N(D),W(D)$ to be the ranks of the east ends, north ends, and west ends of $D$ respectively. We have the following equalities:
\begin{align}
  E(D)&=W(D)-n, \qquad S(D)=N(D)-m, \label{e-EW}\\
  r(D)&=S(D)\uplus W(D)= E(D)\uplus N(D), \label{e-SW}
\end{align}
where $\uplus$ means disjoint union. Equations in  \eqref{e-EW} are obvious; Equations in \eqref{e-SW} follow from the fact that each node of $D$ is either a south end or a west end, and similarly is either an east end or a north end. Note that $0$ is a south end as the node $(0,0)$ and an east end as the node $(m,n)$.
On the other hand, given the rank sequence $r(D)$, one can easily reconstruct the Dyck path $D$. For a rank $r\in r(D)$, it is in $S(D)$ if and only if $r+m\in r(D)$, and similar results hold for $W(D)$, $N(D)$ and $E(D)$.

\subsection{Transformation of lattice paths}
There are two elementary transformations of paths (not necessarily in $\cal F_{m,n}$). Denote by $P^{rev}$ the reverse path of $P$  obtained by reversing the NE-sequence of $P$, without changing the starting and ending point. Geometrically, this corresponds to rotating $P$ by $180$ degrees. The rank sequence of $P$ and $P^{rev}$ are related by
$$ r(P^{rev})=r(\text{ end point of }P)+r(\text{ start point of } P) -r(P).$$

Denote the transpose of $P$ by $P^T = P^{rev} \big|_{N=E,E=N}$. Geometrically, $P^{T}$ is obtained from $P$ by flip along the line $y=-x$ by 180 degrees.
The following two results are immediate. 
\begin{lem}\label{l-mn-switch}
Suppose $m$ and $n$ are coprime. A set $R=\{r_0,\dots, r_{m+n-1}\}$ is the rank sequence of a path $P\in \cal F_{m,n}$ if and only if it is also the rank sequence of
the transpose $P^T\in \cal F_{n,m}$ of $P$.  The result still holds if we restrict to nonnegative rank sequences and Dyck paths.
\end{lem}

\begin{lem}\label{l-R-bar}
Suppose $m$ and $n$ are coprime.
A set $R=\{r_0,\dots, r_{m+n-1}\}$ is the rank sequence of a path $P\in \cal F_{m,n}$ if and only if $-R=\{-r_0,\dots,-r_{m+n-1}\}$ is the rank sequence of the path $P^{rev}\in \cal F_{m,n}$.
\end{lem}

The reversing map is clearly an involution on $\cal F_{m,n}$. Its fixed points are self-reversing paths, or equivalently, palindromic NE-sequences. Then such $P$ can be factored as
$Q^{rev}\varepsilon Q$ for some path $Q$. Since $Q^{rev}$ and $Q$ have the same number of N steps and E steps, $\varepsilon$ is empty if $m$ and $n$ are both even, is N if $n$ is odd, and is E if $m$ is odd. We have
\begin{align}\label{e-self-reversing}
  \# \{P\in \cal F_{m,n} : P=P^{rev} \} = \left\{
                                       \begin{array}{ll}
                                         \displaystyle\binom{\lfloor m/2\rfloor +\lfloor n/2\rfloor}{ \lfloor m/2\rfloor ,\lfloor n/2\rfloor}, & \mbox{ if $m$ or $n$ is even}; \\
                                         0, & \mbox{if $m$ and $n$ are both odd.}
                                       \end{array}
                                     \right.
\end{align}

\subsection{Rank complement of rational Dyck paths}
The situation is a little subtle when restricted to the rank sequences of Dyck paths: $D^{rev}$ has negative ranks. To settle this problem, define the rank complement $\bar{D}$ of $D$ to be the path obtained by first sliding the part before the highest rank so that
the rank $0$ node $(0,0)$ becomes the other rank $0$ node $(m,n)$, then rotating the path by $180$ degrees, and finally shift the path so that the lowest rank becomes $0$.
To be more precise, if we write the NE-sequence of $D$ as $Q_1Q_2$ so that the end point of $Q_1$ has the highest rank, then $\bar D=(Q_2Q_1)^{rev}= Q_1^{rev}Q_2^{rev}$. See Figure \ref{fig:DyckpathTC} for an example.

\begin{figure}[ht]
\begin{center}
\begin{displaymath}
\mbox{\dessin{width=420pt}{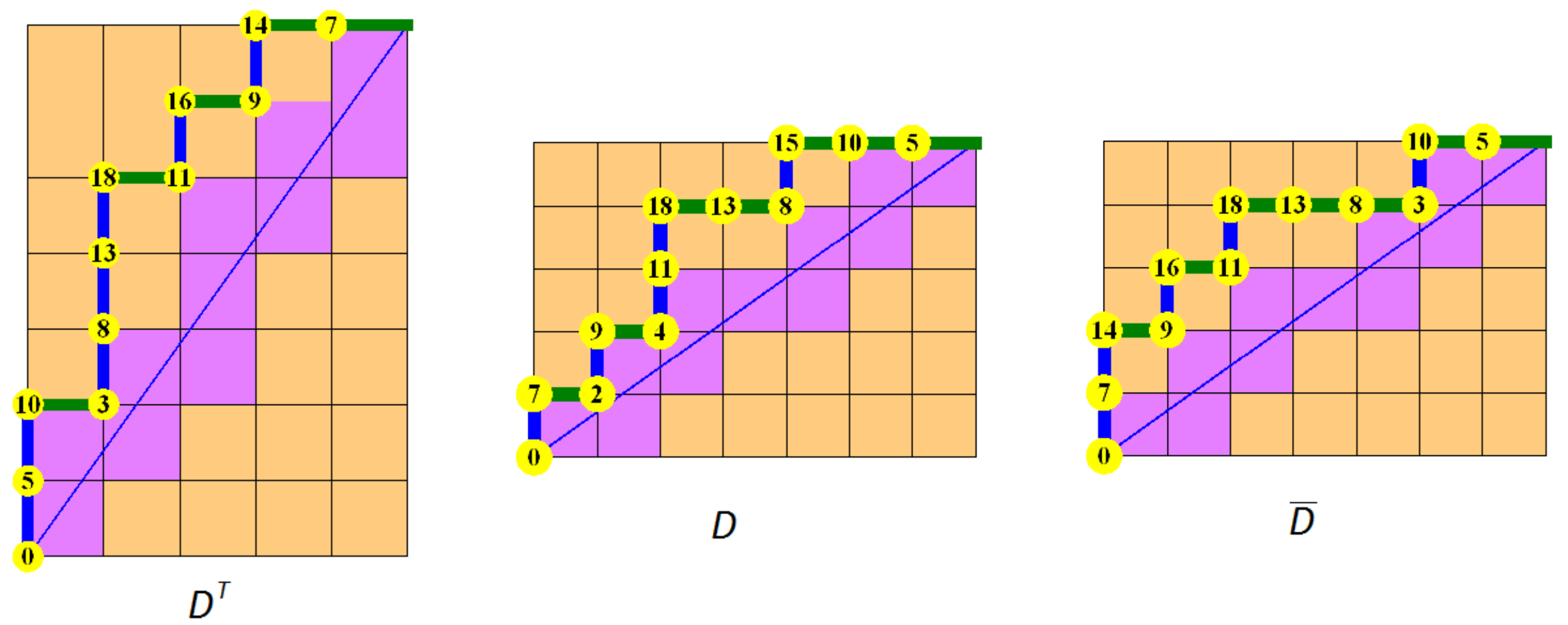}}
\end{displaymath}
\caption{The transpose $D^T$, the Dyck path $D$, and the rank complement $\bar D$.}
\label{fig:DyckpathTC}
\end{center}
\end{figure}

The following result justifies the name of ``rank complement".
\begin{lem}
Suppose $m$ and $n$ are coprime positive integers.
A set $R=\{r_0,\dots, r_{m+n-1}\}$ is the rank sequence of a Dyck path $D$ if and only if its rank complement $\bar{R}=M-R=\{M-r_0,\dots,M-r_{m+n-1}\}$, where $M=\max R$,
is the rank sequence of $\bar D$.
\end{lem}
\begin{proof}
For the sufficient part, we prove the equality $r(\bar D)=\overline{r(D)}$ in two ways. From the geometric construction, we notice that
sliding along the diagonal does not change the ranks, rotating the path gives the rank sequence $-R$ with minimum rank $-M$, and the final shifting will give $-R+M$, as desired.

The alternative proof is by using the formula $\bar D=Q_1^{rev} Q_2^{rev}$, where we split $D$ as $Q_1Q_2$, and regard $Q_2$ and $Q_2^{rev}$ as paths starting at the highest rank $M$. The desired formula follows by observing that $r(Q_i^{rev})=M-r(Q_i)$ for $i=1,2$.

The necessary part follows from the sufficient part and the fact $\bar {\bar{D}}=D$, or equivalently $\bar{\bar R}=R$, which is obtained by $\max \bar{R}=M-(\min R)=M$.
Alternatively, note that $Q_1^{rev}$ ends at the highest rank $M$ of $\bar D$, so that $\bar D=Q_1^{rev}Q_2^{rev}$ is the desired splitting at the highest rank.
\end{proof}

The rank complement transformation is an involution on $\cal D_{m,n}$, i.e., $\bar{\bar D}=D$. When restricted to the ranks of the south ends or east ends, we have the following result.
\begin{lem}\label{l-southendsComplement}
Suppose $m$ and $n$ are coprime and $D$ is an $(m,n)$-Dyck path. Then we have $S(\bar D)=\max(S(D)) -S(D) $ and  $E(\bar D)=\max(E(D)) - E(D)$.
\end{lem}
\begin{proof}
We only prove the equality for south ends. The equality for east ends is similar.

  First observe that $M=\max r(D) =\max S(D)+m=M_s+m$. If $s_i$ is a south end, then $s_i+m \in r(D)$ is a north end of $D$. Thus $M-s_i$ and $M-(s_i+m)$ both belongs to $r(\bar D)$. It follows that
  $M-s_i-m=M_s-s_i$ is a south end of $\bar D$. Together with $\bar {\bar D}=D$, we deduce that $S(\bar D)= M_s-S(D)$, as desired.

  On can also see that south ends of $D$ become north ends of $\bar D$ by the geometric construction of $\bar D$.
\end{proof}

We conclude this section by the following analogous result of Equation \eqref{e-self-reversing}.
\begin{theo}\label{t-self-rank-complement}
  Let $(m,n)$ be a coprime pair. Then the number of self rank complement $(m,n)$-Dyck paths is given by
\begin{align}
  \label{e-self-rank-complement}
\# \{ D \in \cal D_{m,n} : D=\bar D\} = \displaystyle\binom{\lfloor m/2\rfloor +\lfloor n/2\rfloor}{ \lfloor m/2\rfloor ,\lfloor n/2\rfloor}.
\end{align}
\end{theo}
\begin{proof}
We construct a bijection $\Psi$ from the set $A$ of self-rank complementary Dyck paths to the set $B$ of paths with $m'=\lfloor m/2\rfloor$ E steps
 and $n'=\lfloor n/2 \rfloor$ N steps. The theorem then follows since the latter set clearly has cardinality $\binom{m'+n'}{m',n'}$.

It is convenient to allow the use of half N and half E steps. We use the rank system of $\cal F_{m,n}$ through out the proof of the theorem. That is,
a point $(a,b)$ has rank $b m- a n$.


Given a Dyck path $D$ with $\bar D=D$, we split $D$ at its highest rank $2M$ into two palindromic paths.
Then we can write $D=Q_1^{rev} Q_1 Q_2^{rev} Q_2$ where each $Q_i$ may start with a half step. Let $Q=Q_1Q_2^{rev}$. If $Q$ starts with $E^{1/2}$ or end with $N^{1/2}$ (or both), then
set $\Psi(D)$ to be the path $Q$ with the half steps (if any) removed; otherwise, set $\Psi(D)$ to be the path $Q^{rev}$ with the half steps removed.

To see that $\Psi(D) \in B$, we observe that even $m$ and even $n$ can not be coprime, and that half N step appears if and only if $n$ is odd and half E step appears if and only if $m$ is odd. Thus the type of the half steps is determined by the parity of $m$ and $n$.

Conversely, given a path in $B$, adding $N^{1/2}$ at the end if $n$ is odd, and adding $E^{1/2}$ at the beginning if $m$ is odd,
reversing the resulting path if need to get a new path $Q'$ satisfying $\max r(Q')+\min r(Q')>0$. Note that this is always possible
since $r({Q'}^{rev})=-r(Q')$ ($Q'$ starts at $(0,0)$ of rank $0$ and ends at
$(m/2,n/2)$ of rank $0$). Split $Q'$ at the highest rank to get
$Q'=Q'_1 {Q'_2}^{rev}$. Finally set $D'={Q'_1}^{rev} Q'_1 {Q'_2}^{rev} Q'_2$ to be the inverse image.

\begin{figure}[ht]
\begin{center}
\begin{displaymath}
\vcenter{\hbox{\dessin{width=330pt}{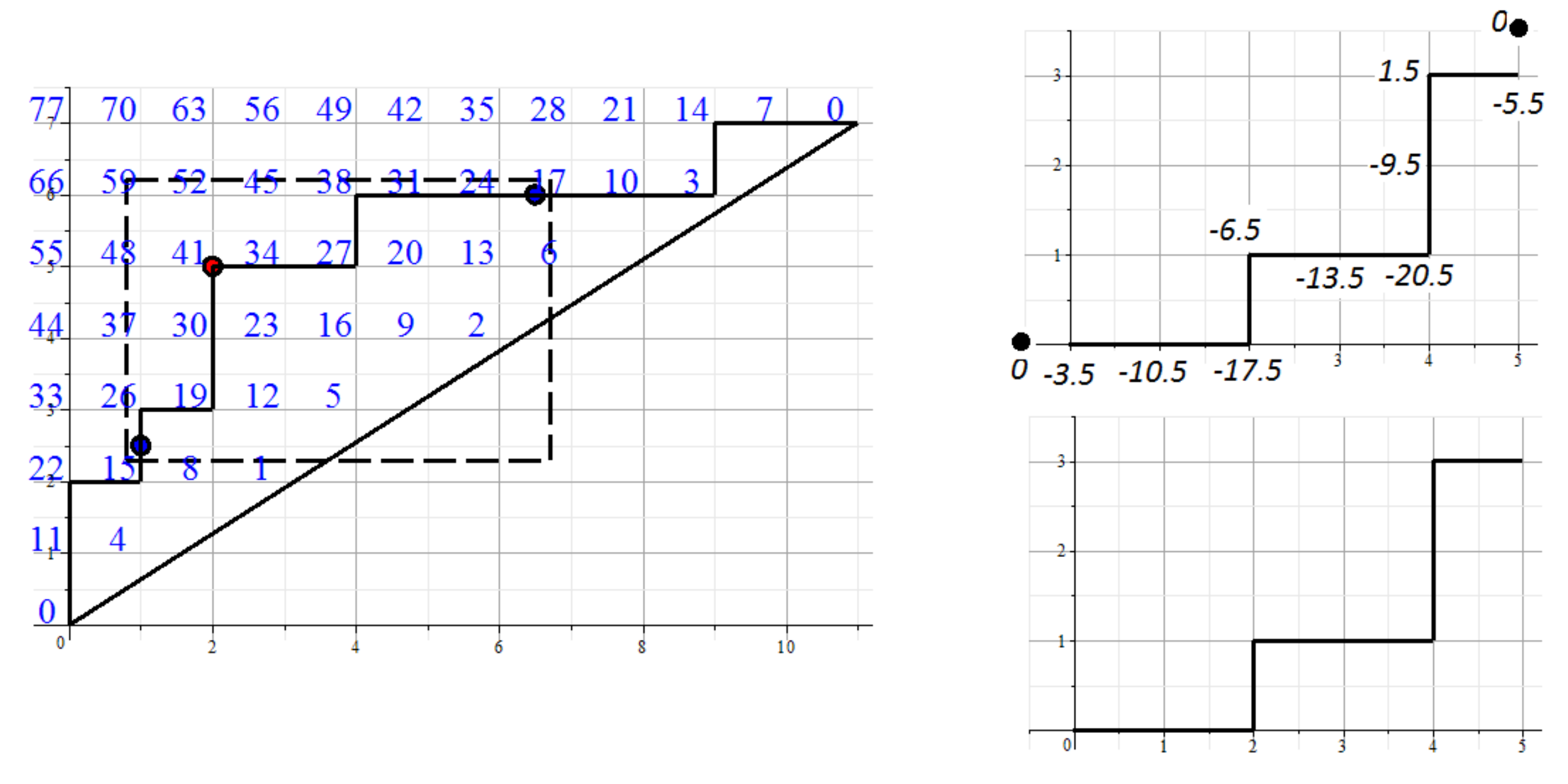}}}
\end{displaymath}
\caption{An example for the bijection $\Psi$, where $(m,n)=(11,7)$. The top right picture is the intermediate step.}
\label{fig:Psi bijection}
\end{center}
\end{figure}

To see that $D'\in A$, we only need to check that $\min r(D')=0$ and $Q'=Q$.
Indeed, if we set $M=\max r(Q')$, then
\begin{align*}
  r({Q'_1}^{rev} Q_1)&=  r({Q'_1}^{rev}) \cup ( M+r(Q_1'))= (M- r(Q'_1)) \cup ( M+r(Q_1')). %
\end{align*}
Thus ${Q_1'}^{rev} Q_1'$ is a path starting at lowest rank $0$ and ending at the highest rank $2M$ if and only if $Q_1'$ is a path starting at rank $0$ ending at highest rank $M$
and having all ranks larger than $-M$. Similarly, ${Q_2'}^{rev}Q_2'$ is a path starting at highest rank $2M$ and ending at lowest rank $0$ if and only if
$Q_2'^{rev}$ is a path starting at the highest rank $M$, ending at rank $0$ and having all ranks larger than $-M$. Therefore $Q'$ satisfies $\max r(Q')+\min r(Q')>0$
if and only if $D'\in A$ and $\max r(D') =2\max r(Q')$.
%
%
%
%
%
\end{proof}

\section{An alternative approach to $n$-core partitions \label{sec:n-core}}
A partition of $\n\in \N$ is a weakly decreasing sequence $\lambda=\lambda_1\ge \cdots\ge \lambda_a>0$ of positive integers such that
$\n=|\lambda|=\lambda_1+\cdots+\lambda_a$. The integer $\n$ is also called the size or area of $\lambda$ and denoted by $\lambda \vdash \n$.
The $i$-th part is $\lambda_i$. The length of $\lambda$ is $\ell(\lambda)=a$, the number of parts of $\lambda$.
The empty sequence is regarded as the unique partition of $0$ with length $0$.
A partition $\lambda$ is usually identified with its Ferrers-Young diagram, which consists of $\lambda_i$ left justified boxes (or nodes, cells) in the $i$-th row from the top. Then a cell $c=(i,j)$ belongs to $\lambda$ if and only if $i\le \ell(\lambda)$ and $j\le \lambda_i$. The hook number $h(c)$ of a cell $c$ is defined by $h(c)=1+\arm(c)+\leg(c)$, where $1$ corresponds to $c$ itself, the arm $\arm(c)$ is the number of cells to the right of $c$, and the leg $\leg(c)$ is the number of cells that are below $c$.
Flipping along the main diagonal gives the Ferrers-Young diagram of the conjugate of $\lambda$, denoted $\lambda^T=(\lambda_1',\dots,\lambda_b')_\ge$. Then $|\lambda^T|=|\lambda|$ and $\lambda_i'$ is just the number of boxes in the $i$-th column of the diagram of $\lambda$. Conjugation transposes $c=(i,j)$ in $\lambda$ to $c'=(j,i)$ in $\lambda^T$, with $\arm(c)=\leg(c')$, $\leg(c)=\arm(c')$ and hence $h(c)=h(c')$.
For example, $\lambda=(5,4,2,1)$ is a partition of $12$ with $4$ parts. Its conjugate $\lambda^T=(4,3,2,2,1)$ has $\lambda_1=5$ parts. See Figure \ref{fig:partition} for an  illustration of the above terminologies.
\begin{figure}[ht]
\begin{center}
\begin{displaymath}
\vcenter{\hbox{\dessin{width=330pt}{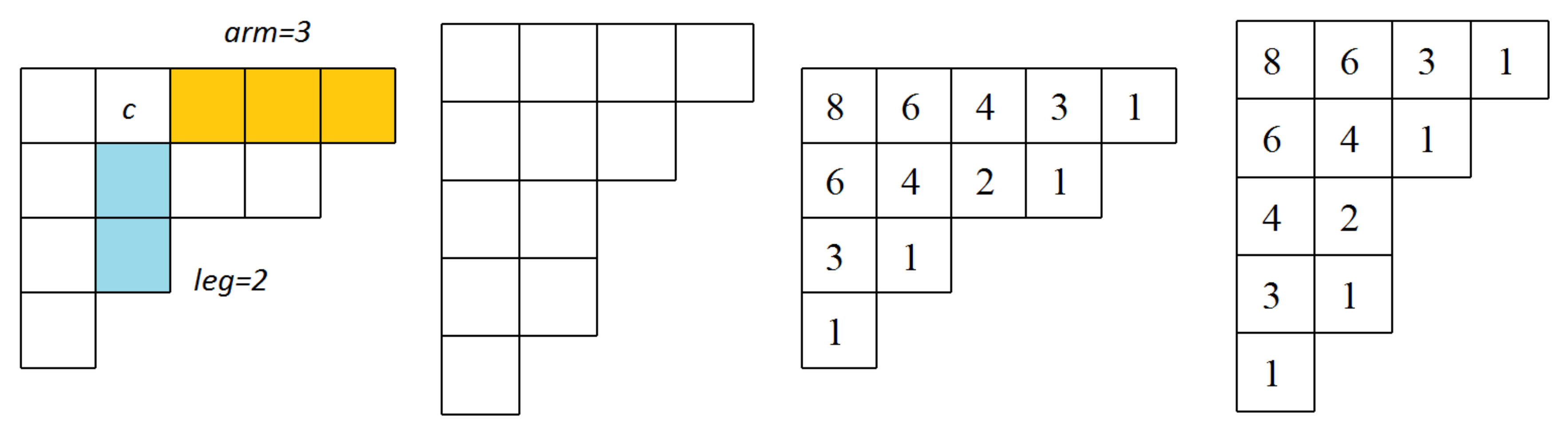}}}
\end{displaymath}
\caption{Partition $\lambda=(5,4,2,1)$, the conjugate $\lambda^T=(4,3,2,2,1)$, $\lambda$ with hook numbers put in the boxes, and similar for $\lambda^T$.}
\label{fig:partition}
\end{center}
\end{figure}

A partition $\lambda$ is said to be an $n$-core if all the hook lengths of $\lambda$ are not divisible by $n$ (or equivalently, equal to $n$). The $\lambda$ in Figure \ref{fig:partition} is a $5$-core.
Core partitions were introduced by Nakayama in the theory of symmetric group representations
\cite{Nakayama1,Nakayama2}.

Denote by $H_\lambda=\{\lambda_1+a-1, \lambda_2+a-2, \dots, \lambda_a\}_>$ the set (arranged decreasingly) of the first column hook lengths of $\lambda$. Then a set of positive integers $H$ uniquely determines
a partition $\lambda$ such that $H=H_\lambda$.

Given a positive integer $n$ and a set $H=\{h_1,\dots, h_a\}_>$ of positive integers,
define
$$s_i=s_i^n(H)=\max ( (n+H) \cup \{i\} \cap (n\mathbb{Z}+i)), \qquad \text{for } 0\le i\le n-1,$$
and let $S^n(H)=\{s_0,\dots, s_{n-1}\}$. (We will often omit the superscript $n$ when it is clear from the context.)
In words, if there exists an $h\in H$ with $h \equiv i \bmod{n}$ then $s_i-n$ is the maximum of such $h$, otherwise $s_i=i$.
Clearly, we have $\max S(H)=h_1+n$.
We say $H$ is $n$-flush at $i$ if $H \cap (n\Z+i)$ is
either empty (hence $s_i=i$) or $\{i,n+i,\dots,s_i-n\}$. The set $H$ is called $n$-flush if it is $n$-flush at $i=1,2,\dots,n-1$ and $s_0=0$.
In other words, every $h\in H$ is not divisible by $n$ and $h\in H$ implies that $h-n$ is either negative or belongs to $H$.
Thus $S(H)$ is a natural encoding of the $n$-flush set $H$. Later we will see that in a different context,
$S(H)$ corresponds to the ranks of the south ends of a Dyck path.

\begin{theo}\label{t-SH-complement}
Suppose the first column hook lengths $H_\lambda=\{h_1,\dots, h_a\}_>$ of a partition $\lambda$ is an $n$-flush set. Then $H_{\lambda^T}$ is also $n$-flush and
\begin{align}
  S(H_\lambda) = h_1+n -S(H_{\lambda^T}). \label{e-SHH'}
\end{align}
\end{theo}
\begin{proof}
We prove the the theorem by induction on the length $a$ of $\lambda$.
Equation \eqref{e-SHH'} clearly holds when $\lambda$ is empty. If $a=1$ then $h_1=\lambda_1<n$ for otherwise $h_1-n\in H_\lambda$, a contradiction.
Now $S(H_\lambda)=\{0,1,\dots,n-1\}\setminus \{h_1\} \cup \{h_1+n\}$ and
\begin{multline*}
  S(H'_\lambda)=S(\{h_1,h_1-1,\dots,1\})=\{0,n+1,\dots, n+h_1,h_1+1,\dots,n-1\}\\
  =\{h_1+1,\dots,h_1+n\}\setminus \{n\}\cup \{0\}
=h_1+n-(\{0,\dots,n-1\}\setminus \{h_1\}\cup \{h_1+n\}  ),
\end{multline*}
which is just $h_1+n-S(H_\lambda)$  as desired. It is also clear that $\{h_1\}$ is $n$-flush if and only if $h_1<n$,
   if and only if $\lambda=(h_1)_>$ is an $n$-core.

Now assume $a\ge 2$ and we proceed the induction on $a$ for the theorem. Denote by $H'=H'_\lambda=\{h_1',\dots, h'_b\}_< $ the first row hook lengths of $\lambda$. This is also the first column hook lengths of $\lambda^T$.
Then $b=\lambda_1$ is the length of $\lambda^T$. Clearly $h_1+n=h_1'+n$ is the maximum of both $S(H_\lambda)$ and $S(H'_\lambda)$.
Let $\mu$ be the partition obtained from $\lambda$ by removing the first row.
Then the hook lengths of the first column of $\mu$ is just $H_\mu=\{h_2,\dots,h_a\}_<$, which is also $n$-flush.
Since $H_\lambda$ is $n$-flush, $h_1-n$ is either negative or belongs to $H_\lambda$. It follows that $h_2\ge h_1-n$, or equivalently $h_1-h_2\le n$, and
that $S(H_\lambda)$ and $S(H_\mu)$ differ by only one element, i.e., $S(H_\lambda)$ has $h_1+n$ but $S(H_\mu)$ has $h_1$.
In formula we have
\begin{align}\label{e-SHlambdamu}
  S(H_\lambda)=S(H_\mu)\setminus \{h_1\}\cup \{h_1+n\}.
\end{align}

Let $G'=H'_{\mu}=\{g'_1,\dots, g_c'\}_>$ be the first row hook lengths of $\mu$. Then $c=\lambda_2$ and $g_1'=h_2$.
By the induction hypothesis, $G'$ is $n$-flush and Equation \eqref{e-SHH'} holds for $\mu$:
\begin{align}
  \label{e-SHH'mu}
  S(H_\mu) = h_2+n -S(H'_{\mu})=h_2+n-S(G').
\end{align}
In particular $h_1\in S(H_\mu)$ implies that $h_2+n-h_1\in S(G')$, that is, $s_{n-h_1+h_2}(G')=n-h_1+h_2$.

Claim: the set $H'_{\lambda}$ is $n$-flush, and the following equation holds.
\begin{align}\label{e-SH'lambdamu}
  S(H'_\lambda)=(S(H'_\mu)+h_1-h_2)  \setminus \{n\} \cup \{0\}.
\end{align}
Assuming the claim holds, we can complete the induction by the following computation.
\begin{align*}
  h_1+n-S(H_\lambda) &= h_1+n- (S(H_\mu) \setminus \{h_1\}\cup \{h_1+n\}) \qquad (\text{by \eqref{e-SHlambdamu}}) \\
                     &=(h_1+n-S(H_\mu)) \setminus \{n\} \cup \{0\} \\
(\text{by \eqref{e-SHH'mu}})    \qquad                 &=(h_1-h_2+S(H'_\mu) ) \setminus \{n\} \cup \{0\}  \\
        \text{(by \eqref{e-SH'lambdamu})}       \qquad      &=S(H'_\lambda).
\end{align*}

To show the claim, observe that $h_j'-g_j'=b-c+1$ for $j=1,2,\dots, c$ and
$h_{c+1}',\dots, h_b'$ are given by $h_1'-g_1'-1,\dots, 2,1$. In formula we have
$$  H'= (b-c+1 +G') \cup \{1,2,\dots,b-c\}= (\delta +G') \cup \{1,2,\dots,b-c\},$$
where we have set $\delta= b-c+1 =h_1-h_2$ for convenience.

Since $G'$ is $n$-flush, it is $n$-flush at $i$ for each $i$. Thus $G' \cap (i+n\Z)=\{i,n+i,\dots, s_i(G')-n\}$, where this set is taken to be empty if $s_i(G')<n$. We have
  \begin{align*}
    H'\cap (\delta +i+n\Z)&= ((\delta +G') \cup \{1,2,\dots,b-c\} )\cap (\delta+i+n\Z)\\
                           &= (\delta+ (G'\cap (i+n\Z))) \cup (\{1,2,\dots,b-c\} \cap (\delta+i+n\Z))\\
                           &=\{i+\delta,i+n+\delta,\dots, s_i(G')-n+\delta \} \cup A,
  \end{align*}
where $A$ is $\{\delta+i-n\}$ if $\delta+i>n$
and is empty otherwise. In the former case, since $h_1-h_2 \le n$, we have $0<\delta+i-n<n$ so
that $H'$ is flush at $\delta+i-n$ with $\delta+s_i(G')\in S(H')$;
In the latter case, if $\delta+i<n$ then $H'$ is flush at $\delta+i$ with $\delta+s_i(G')\in S(H')$;
in the only remaining case $\delta+i=n$, by the assumption $s_{n-\delta }(G')=n-\delta <n$, we have $H'\cap (n+n\Z)=\emptyset$ with
$\delta+s_{n-h_1+h_2}(G')=n \not\in S(H')$ but $s_0(H')=0$ instead. This completes the proof since when $i$ ranges over $0,1,\dots,n-1$,
so does $h_1'-g_1'+i \bmod n$.
\end{proof}

Theorem \ref{t-SH-complement} establishes a simple connection between $H_\lambda$ and $H_{\lambda^T}$ for $n$-core $\lambda$. It is an important tool in
our study of $(m,n)$-cores. One consequence of Theorem \ref{t-SH-complement} is the following classical result. This result has an elegant proof
 using the abacus notation of partitions. See, e.g., \cite{Anderson}. As a warm up and also for self-containedness, we include a new proof of this classical result.

\begin{cor}
A partition $\lambda$ is an $n$-core if and only if  $H_\lambda$ is $n$-flush.
\end{cor}
\begin{proof}

The necessity follows by the fact that the $n$-flushness of $H_\lambda$ implies the $n$-flushness of $H_\mu$ and $H_{\lambda^T}$.
Thus inductively the hook lengths of each row of $\lambda$ is $n$-flush, and hence contains no multiple of $n$.

We prove the sufficiency part of the theorem by contradiction.
We follow notations in the proof of Theorem \ref{t-SH-complement}. Assume $a$ is minimal such that $\lambda$ is an $n$-core of length $a$ but $H_\lambda$ is not $n$-flush. Clearly we may assume $a\ge 2$. Use the notation for $\lambda$ and $\mu$ as before. By the minimality of $a$, $H_\mu$ is $n$-flush. If $b-c=h_1-h_2-1\ge n$ then $H_\lambda'$ has $n$ as a hook length of $\lambda$. A contradiction. So assume $h_1-h_2\le n$. By definition $h_1$ is not a multiple of $n$.
Let $j$ be the smallest number such that $h_1-jn$ is negative or belongs to $H_\mu$. Then $j\ge 2$ since $H_\lambda$ is not $n$-flush but $\mu$ is $n$-flush.

By Theorem \ref{t-SH-complement},
\eqref{e-SHH'mu} holds. It follows that
$h_2+n-(h_1-jn+n)=h_2-h_1+jn \in S(H_\mu')$, so that $jn+h_2-h_1-n>0$ appears as a hook length $g_k'$ in $H_\mu'$  for some $k\le c$. Then $g_k'+h_1-h_2=jn+h_2-h_1-n+h_1-h_2=jn-n$ appears as the hook length $h_k'$, contradicting the assumption that $\lambda$ is an $n$-core partition.
\end{proof}


\section{Conjugation of $(m,n)$-cores and rank complement of $\cal D_{m,n}$ \label{sec:conjugation-rank-complement} }

A partition $\lambda$ is said to be an $(m,n)$-core if it is both an $m$-core and $n$-core. By Theorem \ref{t-SH-complement}, $\lambda$ is $(m,n)$-core if and only if $H_\lambda$ is $(m,n)$-flush, i.e., $H_\lambda$ is both $m$-flush and $n$-flush.
When $m$ and $n$ are relatively prime, $h\in H_\lambda$ can be uniquely written as $h=am-bn$ for $0\le b<m$. We claim that $a<n$, so that $h$ is the rank of a lattice point in the $m$ by $n$ rectangle. If $a\ge n$ then we can write $h=(a-n)m+(m-b)n$. Then the $n$-flushness of $H_\lambda$ implies that $a\ne n$ and that $h-(m-b)n=(a-n)m$ is also in $H_\lambda$, which contradicts the $m$-flushness of $H_\lambda$.

Now for a coprime pair $(m,n)$, we identify $H_\lambda$ with the ranks of lattice points $V$ inside the $m$ by $n$ rectangle. Then the $(m,n)$-flush property says that $h\in H_\lambda$ implies that $h-m$ and $h-n$, if positive, are also in $H_\lambda$. This corresponds to that if $(a,b)\in V$ then $(a+1,b)$ and $(a,b-1)$ are also in $V$, provided that they are above the diagonal. Such $V$ are exactly the lattice points to the right of an $(m,n)$-Dyck path $D$. This gives Anderson's bijection that takes $D$ to $\alpha(D)=\lambda$ by setting $H_\lambda$ to be the
the positive ranks of the lattice points to the right of the $(m,n)$-Dyck path $D$. The set $S(H_\lambda)$ is exactly the ranks of the south ends of
$D$. See Figure \ref{fig:anderson bijection} for examples.
\begin{figure}[ht]
\begin{center}
\begin{displaymath}
\vcenter{\hbox{\dessin{width=430pt}{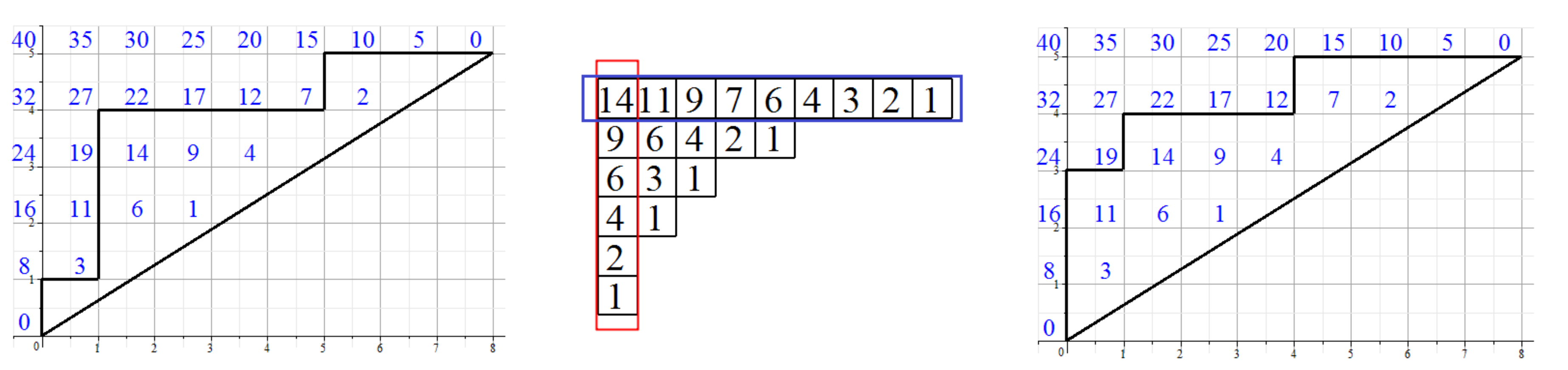}}}
\end{displaymath}
\caption{Anderson's bijection for $(m,n)=(8,5)$: Dyck path $D$, $\lambda=\alpha(D)$, 
and the rank complement $\bar D$. The first column hook lengths of $\lambda$ are the ranks to the right of $D$, and the first row hook lengths of $\lambda$ are the ranks to the right of $\bar D$.}
\label{fig:anderson bijection}
\end{center}
\end{figure}


If a partition $\lambda$ is an $(m,n)$ core, then clearly so is its
conjugate $\lambda^T$. A natural question is what is the relation between the paths $\alpha^{-1}(\lambda)$ and $\alpha^{-1} (\lambda^T)$. See Figure \ref{fig:anderson bijection} for an example. At the first glance, the two paths seems very different. But we have the following result.
\begin{theo}\label{t-mncore-conjugate}
Suppose $m$ and $n$ are coprime and $D$ is an $(m,n)$-Dyck path. Then the $(m,n)$-core $\alpha(\bar D)$ is just the conjugate of the $(m,n)$-core $\alpha (D)$.
\end{theo}
\begin{proof}
Denote by $\lambda=\alpha(D)$ and $\lambda'=\alpha(\bar D)$. Then $\max S(D)=\max S(\bar D)=h_1(\lambda)+n=h_1(\lambda')+n$. By Lemma \ref{l-southendsComplement},
$$S(H_{\lambda'})=S(\bar D)= h_1+n-S(D)=h_1+n-S(H_{\lambda}), $$
where $H_{\lambda}$ and $H_{\lambda'}$ are regarded as $n$-flush set.

By Theorem \ref{t-SH-complement}, we also have
$S(H_{\lambda^T})=h_1+n-S(H_\lambda)$. It follows that $S(H_{\lambda^T})=S(H_{\lambda'})$, and hence $H_{\lambda^T}=H_{\lambda'}$ and $\lambda^T=\lambda'$.
This completes the proof.
\end{proof}

Combining Theorems \ref{t-mncore-conjugate} and \ref{t-self-rank-complement}, we recover Theorem \ref{c-num-self-conjugate-mncore}.

%
%
%

\section{The dinv statistic on $\cal D_{m,n}$ and the skew-length for $(m,n)$-cores \label{sec:dinv-skew-length}}
\subsection{The area and dinv statistics for rational Dyck paths}
We introduce two important statistics on an $(m,n)$-Dyck path $D\in \cal D_{m,n}$ for general pair $(m,n)$ of positive integers.
The area $\area(D)$ of $D$ is the number of lattice boxes between $D$ and the diagonal line $y=\frac{n}{m}x$.
The dinv statistic arose in the theory of parking functions. When restricted to
Dyck paths, $\dinv(D)$ of $D$ specializes to the following geometric construction: It counts the number of
cells $c$ of the partition above the path whose arm and leg satisfy the inequalities
$$ \frac{\arm(c)}{\leg(c) + 1} \le \frac{m}{n} < \frac{\arm(c) + 1}{\leg(c)}. $$
In Figure \ref{fig:DinvT} we have placed a green square in each of the cells that contribute to the $\dinv(D)$ (in the left picture) and
$\dinv(D^T)$ (in the right picture).
\begin{figure}[ht]
\begin{center}
\begin{displaymath}
\vcenter{\hbox{\dessin{width=330pt}{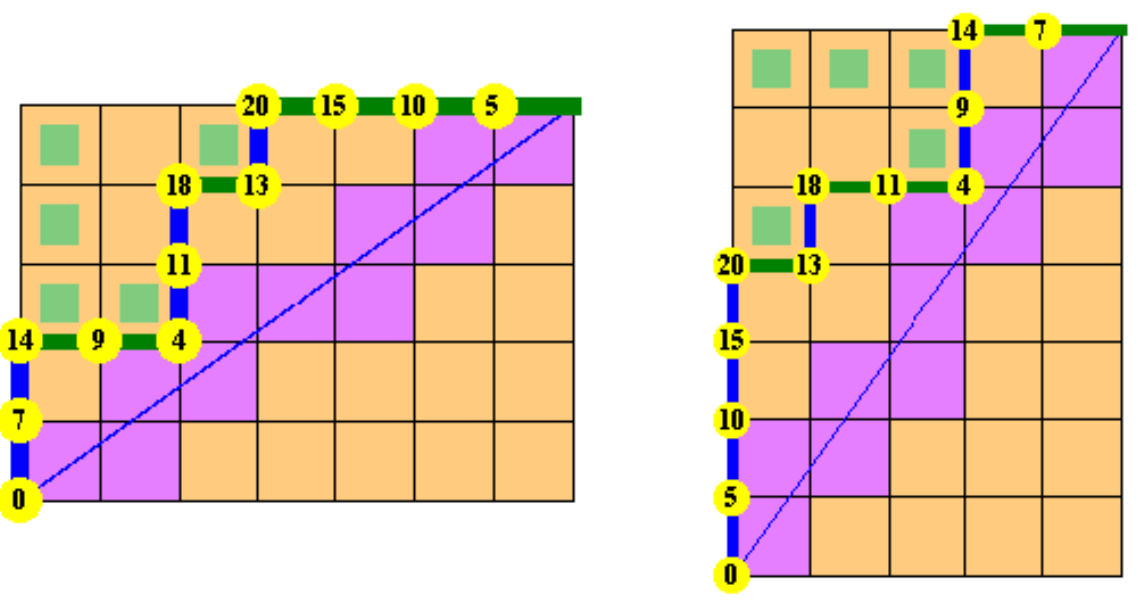}}}
\end{displaymath}
\caption{The dinv statistic and transpose of Dyck paths}
\label{fig:DinvT}
\end{center}
\end{figure}

An immediate consequence of the geometric construction is the following.
\begin{prop}
For every rational Dyck path $D\in \cal D_{m,n}$, we have
$$\dinv D= \dinv D^T. $$
\end{prop}

The dinv statistic also appeared as the $h^+$ statistic. See, e.g., \cite{Amstrong-Catalan}. Alternative descriptions of
dinv have been found useful. The one in \cite[Lemma~11]{Amstrong-Catalan} can be stated as follows:
If $D$ is encoded by a sequence $(P_0,\dots, P_{m+n})$ of lattice points, then we have
\begin{align}
  \dinv(D) = \# \{i<j:  P_i\in W(D), P_j \in S(D), \text{ and } 0<r(P_i)-r(P_j) \le m+n\}.
\end{align}

Now we concentrate on the case when $(m,n)$ is a coprime pair. Our language translate the dinv statistic as follows.
\begin{align}
  \dinv(D) = \# \{ 0<r_w-r_s \le m+n : r_w\in W(D), r_s \in S(D), 
 r_w \text{ proceeds 
} r_s 
 \},
\end{align}
where ``proceeds" refers to the corresponding lattice points of $r_w$ and $r_s$ in $D$.


For a coprime pair $(m,n)$, consider the $m\times n$ rectangle. The number of lattice boxes to the left of the diagonal is $(m-1)(n-1)/2$, as the diagonal
hits a ribbon of $m+n-1$ boxes (the purple boxes in Figure \ref{fig:DinvT}), and the ribbon divides the rest of the rectangle into two equal pieces.
Thus
the maximum value of $\area(D)$ over $D\in \cal D_{m,n}$ is $\frac{(m-1)(n-1)}{2} $. This is also the maximum of $\dinv(D)$ by the geometric construction.
It is natural to use the notations
$$\coarea(D)=\frac{(m-1)(n-1)}{2}-\area(D), \qquad \codinv(D)=\frac{(m-1)(n-1)}{2}-\dinv(D).$$


\begin{defn}[The Sweep Map]
Suppose a Dyck path $D \in \cal D_{m,n}$ is encodes by its NE-sequence $D=p_1\cdots p_{m+n}$. Then the sweep map $\Phi(P)$ is obtained by sorting the steps $p_i$ into increasing order according to the ranks of their starting points.
\end{defn}
For example, take $(m,n)=(7,5)$, and $P=NNEENNENEEEE$. Then the path and the ranks are computed as
$$ \left[ \begin {array}{cccccccccccc} N&N&E&E&N&N&E&N&E&E&E&E\\
 0&7&14&9&4&11&18&13&20&15&10&5
\end {array} \right]. $$
Sorting the path according to the ranks gives
$$ \left[ \begin {array}{cccccccccccc} N&N&E&N&E&E&N&N&E&E&E&E
\\ 0&4&5&7&9&10&11&13&14&15&18&20\end {array}
 \right],
$$
and thus $\Phi(P)=N N E N E E N N E E E E$.

It is true but not obvious that $\Phi(D)$ is still in $\cal D_{m,n}$. See \cite{sweepmap} for a proof. The dinv statistic is closely related to the $\Phi$ map by the following result.
\begin{theo}\label{t-Phi-dinv}
Let $(m,n)$ be a coprime pair. Then for an $(m,n)$-Dyck path $D$, we have
$\dinv(D)=\area( \Phi(D))$, and consequently,
$$  \codinv(D)= \coarea ( \Phi(D))= \# \{(r_s,r_w) :   r_s\in S(D), r_w \in W(D), r_s<r_w \}.   $$
\end{theo}
Theorem \ref{t-Phi-dinv} holds for general $m$ and $n$ with certain modifications. This is one of the results in \cite[Theorem 16]{Loehr-Warrington}.
A bijective proof of $\codinv(D)= \coarea ( \Phi(D)) $ can be found in \cite{Gorsky-Mazin}.

\begin{cor}
 Let $(m,n)$ be a coprime pair. Then the rank complement transformation preserves the dinv statistic. In other words,
for any $(m,n)$-Dyck path $D$, we have
$$\dinv (\bar D)= \dinv(D).$$
\end{cor}
\begin{proof}
  It is sufficient to show that
$\codinv(\bar D)= \codinv(D)$. Recall that $r(\bar D)=M-r(D)$, where $M=\max r(D)$. We have
\begin{align*}
  \codinv(\bar D) &= \# \{(\bar r_s,\bar r_w) :   \bar r_s\in S(\bar D), \bar r_w \in W(\bar D), \bar r_s<\bar r_w \} \\
&=\# \{(r_n,r_e) :   r_n\in N(D), r_e \in E(D), M-r_n<M-r_e \}  \qquad (\text{work in } D)\\
&=\# \{(r_e,r_n) :   r_n\in N(D), r_e \in E(D), r_e<r_n \} \\
&=\# \{(r_s,r_w) :   r_w\in W(D^T), r_s \in S(D^T), r_s<r_w \} \qquad (\text{work in } D^T) \\
&=\codinv(D^T) =\codinv(D).
\end{align*}
\end{proof}

\begin{rem}
  Sweep map has become an active subject in the recent 15 years. Variations and extensions has been found, and some classical bijections turns out to be disguised sweep map. See \cite{sweepmap}. One major conjecture in this area, in its simplest form, is that the $\Phi$ map is a bijection from the set $\cal D_{m,n}$ of Dyck paths to itself. The conjecture has only been shown for some special cases.
\end{rem}

\subsection{Skew-length of $(m,n)$-cores}

For an $(m,n)$-core $\lambda$, its skew length $s\ell^{m,n}(\lambda)$ (the superscript $m,n$ will be omitted if it is clear from the context) was defined in \cite{Armstrong-mn-core} to be the number of boxes that are both in the $m$ boundaries and in the $n$-rows.
The $m$ boundaries are the set of boxes whose hook length is less than $m$; the $n$-rows is better described in our language: if $H=H_\lambda=(h_1,\dots,h_a)_>$
is the first column hook lengths of $\lambda$, then the $i$-th row is called an $n$-row if $h_i+n\in S^n(H)$ for the
$n$-flush set $H$.

The following is an example of an $(8,5)$-core $\lambda=(9,5,3,2,1,1)$. The hook lengths of $\lambda$ is put to the right of $||$. The left most column are $h_i \bmod 5$ and we have underlined the $5$-rows, and bold faced  the $8$-boundaries. Thus $s\ell(\lambda)=6+3+1=10$.
$$\begin{array}{c||ccccccccc}
4=(14 \bmod 5) &14 &11&9&\mathbf{7}&\mathbf{6}&\mathbf{4}&\mathbf{3}&\mathbf{2}&\mathbf{1}\\\hline
4=(9 \bmod 5) &9 &\mathbf{6}&\mathbf{4}&\mathbf{2}&\mathbf{1}&&&&\\
1=(6 \bmod 5) &\mathbf{6} &\mathbf{3}&\mathbf{1}&&&&&&\\\hline
4=(4 \bmod 5) &\mathbf{4} &\mathbf{1}&&&&&&&\\
{2}=(2 \bmod 5) &\mathbf{2} &&&&&&&&\\\hline
1=(1 \bmod 5) &\mathbf{1} &&&&&&&&\\
\end{array}
 $$

\begin{theo}\label{t-skewlength-codinv}
  Let $m$ and $n$ be coprime positive integers. Then for any $(m,n)$-Dyck path $D$ we have
$$ \codinv(D)=s\ell( \alpha(D)).$$
Consequently, the skew length statistic of an $(m,n)$-core $\lambda$ is invariant under conjugation and switching the roles of $m$ and $n$. In formula, we have
$$ s\ell^{m,n}(\lambda)=s\ell^{n,m}(\lambda) =s\ell^{m,n}(\lambda^T)=s\ell^{n,m}(\lambda^T).$$
\end{theo}

To prove Theorem \ref{t-skewlength-codinv}, we need a better description of $s\ell(\lambda)$. Let $\lambda=\lambda_1\ge \cdots\ge \lambda_a>0$ be an $(m,n)$-core with
$H=H_\lambda=\{h_1,\dots, h_a\}_>$. Denote by $\lambda^i=\lambda_i\ge \cdots\ge \lambda_a$ the partition obtained from $\lambda$ by removing the first $i-1$ rows. Then $\lambda^1=\lambda$. Denote by $H^i=H(\lambda^i)=(h_i,\dots, h_a)$. Then the $i$-th row hook lengths of $\lambda$ is just the fist row hook lengths of $\lambda^i$, which is
$H_{(\lambda^i)^T}$.

\begin{proof}[Proof of Theorem \ref{t-skewlength-codinv}]
The number of boxes in row $i$ and also in the $m$-boundaries is:
\begin{align*}
 | H_{(\lambda^i)^T} \cap \{1,2,\dots, m-1\} |
                &=  | S^m\left( H_{(\lambda^i)^T}\right) \cap \Z_{>m} | \\
   \text{(by Theorem \ref{t-SH-complement}) }           &=  |\left( h_i+m - S^m(H^i) \right) \cap \Z_{>m} | \\
 &=  |\left( h_i - S^m(H^i) \right) \cap \Z_{>0} | \\
&=  |\left( S^m(H^i)-h_i \right) \cap \Z_{<0} | \\
&= |\left( S^m(H)-h_i \right) \cap \Z_{<0} |,
\end{align*}
where in the last step, we use the $m$-flushness of $H$, which implies that (recall that $s^m_j(H^i)-m\le h_i$ by definition)
$$H^i \cap (j+m \Z)= \{ s^m_j(H^i)-m, \dots, m+j,j\} = \{ s^m_j(H)-m,\dots, m+j,j\} \cap \Z_{\le h_i},$$
and hence  $s^m_j(H^i)<h_i$ if and only if $s^m_j(H)<h_i$.

 If $\lambda= \alpha(D)$ then $S^n(H_\lambda)$ is just $S(D)$ and $S^m(H_\lambda)$ is just $E(D)$. With these arguments handy, we derive that
\begin{align*}
  s\ell(\lambda)&= \sum_{h_i+n \in S^n(H)}  | H{(\lambda^i)^T} \cap \{1,2,\dots, m-1\} | \\
&=\sum_{h_i+n \in S^n(H)} |\left( S^m(H)-h_i \right) \cap \Z_{<0} |\\
&= | \{(r_s,r_e) :   r_s\in S(D), r_e \in E(D), r_s-r_e-n <0 \}|\\
&= | \{(r_s,r_w) :   r_s\in S(D), r_w \in W(D), r_s-r_w <0 \} |.
\end{align*}
This is exactly the coarea of $\Phi(D)$, which is equal to $\codinv(D)$ by Theorem \ref{t-Phi-dinv}.
\end{proof}

\medskip
\noindent
{\small \textbf{Acknowledgements:} This work was done during the author's visiting at UCSD. The author is very grateful to Professor Adriano Garsia for
inspirations and encouraging conversations.
%
This work was partially supported by the National Natural Science Foundation of China (11171231).}

\end{document}